\documentclass[10pt]{amsart}
\usepackage{epsfig,amsfonts,amsthm,amssymb,latexsym,amsmath,enumitem}

\newcommand{\mymod}[3]{#1 \equiv #2 \kern -0.5em \pmod{#3}}
\newcommand{\mynotmod}[3]{#1 \not \equiv #2 \kern -0.6em \pmod{#3}}

\theoremstyle{plain}
\newtheorem{theorem}{Theorem}[section]
\newtheorem{corollary}[theorem]{Corollary}
\newtheorem{lemma}[theorem]{Lemma}

\theoremstyle{remark}

\theoremstyle{definition}

\title[The Gelin-Ces\`aro identity in third-order Jacobsthal sequences]{The Gelin-Ces\`aro identity in some third-order Jacobsthal sequences}
\vspace{5pt}

\author[G. Cerda-Morales]{\scriptsize GAMALIEL CERDA-MORALES$^{1}$}
\date{}

\begin{document}
\maketitle

\vspace{-20pt}
\begin{center}
{\footnotesize $^1$Instituto de Matem\'aticas, Pontificia Universidad Cat\'olica de Valpara\'iso, \\
Blanco Viel 596, Valpara\'iso, Chile.\\
E-mails: gamaliel.cerda@usm.cl / gamaliel.cerda.m@mail.pucv.cl
}\end{center}

\hrule

\begin{abstract}
In this paper, we deal with two families of third-order Jacobsthal sequences. The first family consists of generalizations of the Jacobsthal sequence. We show that the Gelin-Ces\`aro identity is satisfied. Also, we define a family of generalized third-order Jacobsthal sequences $\{\mathbb{J}_{n}^{(3)}\}_{n\geq 0}$ by the recurrence relation $$\mathbb{J}_{n+3}^{(3)}=\mathbb{J}_{n+2}^{(3)}+\mathbb{J}_{n+1}^{(3)}+2\mathbb{J}_{n}^{(3)},\ n\geq0,$$ with initials conditions $\mathbb{J}_{0}^{(3)}=a$, $\mathbb{J}_{1}^{(3)}=b$ and $\mathbb{J}_{2}^{(3)}=c$, where $a$, $b$ and $c$ are non-zero real numbers. Many sequences in the literature are special cases of this sequence. We find the generating function and Binet's formula of the sequence. Then we show that the Cassini and Gelin-Ces\`aro identities are satisfied by the indices of this generalized sequence.
\end{abstract}

\medskip
\noindent

\medskip
\noindent
\medskip

\hrule

\section{Introduction and Preliminaries}\label{sec:1}
\setcounter{equation}{0}

The Jacobsthal numbers have many interesting properties and applications in many fields of science (see, e.g., \cite{Ba,Hor2,Hor3}). The Jacobsthal numbers $J_{n}$ are defined by the recurrence relation
\begin{equation}\label{e1}
J_{0}=0,\ J_{1}=1,\ J_{n+2}=J_{n+1}+2J_{n},\ n\geq0.
\end{equation}
Another important sequence is the Jacobsthal-Lucas sequence. This sequence is defined by the recurrence relation $j_{n+2}=j_{n+1}+2j_{n}$, where $j_{0}=2$ and $j_{1}=1$ (see, \cite{Hor3}).

In \cite{Cook-Bac} the Jacobsthal recurrence relation is extended to higher order recurrence relations and the basic list of identities provided by A. F. Horadam \cite{Hor3} is expanded and extended to several identities for some of the higher order cases. For example, the third-order Jacobsthal numbers, $\{J_{n}^{(3)}\}_{n\geq0}$, and third-order Jacobsthal-Lucas numbers, $\{j_{n}^{(3)}\}_{n\geq0}$, are defined by
\begin{equation}\label{e2}
J_{n+3}^{(3)}=J_{n+2}^{(3)}+J_{n+1}^{(3)}+2J_{n}^{(3)},\ J_{0}^{(3)}=0,\ J_{1}^{(3)}=J_{2}^{(3)}=1,\ n\geq0,
\end{equation}
and 
\begin{equation}\label{e3}
j_{n+3}^{(3)}=j_{n+2}^{(3)}+j_{n+1}^{(3)}+2j_{n}^{(3)},\ j_{0}^{(3)}=2,\ j_{1}^{(3)}=1,\ j_{2}^{(3)}=5,\ n\geq0,
\end{equation}
respectively.

Some of the following properties given for third-order Jacobsthal numbers and third-order Jacobsthal-Lucas numbers are used in this paper (for more details, see \cite{Cer,Cer1,Cook-Bac}). Note that Eqs. (\ref{e7}) and (\ref{e12}) have been corrected in this paper, since they have been wrongly described in \cite{Cook-Bac}.
\begin{equation}\label{e4}
3J_{n}^{(3)}+j_{n}^{(3)}=2^{n+1},
\end{equation}
\begin{equation}\label{e5}
j_{n}^{(3)}-3J_{n}^{(3)}=2j_{n-3}^{(3)},\ n\geq3,
\end{equation}
\begin{equation}\label{ec5}
J_{n+2}^{(3)}-4J_{n}^{(3)}=\left\{ 
\begin{array}{ccc}
-2 & \textrm{if} & \mymod{n}{1}{3} \\ 
1 & \textrm{if} & \mynotmod{n}{1}{3}
\end{array}%
\right. ,
\end{equation}
\begin{equation}\label{e6}
j_{n}^{(3)}-4J_{n}^{(3)}=\left\{ 
\begin{array}{ccc}
2 & \textrm{if} & \mymod{n}{0}{3} \\ 
-3 & \textrm{if} & \mymod{n}{1}{3}\\ 
1 & \textrm{if} & \mymod{n}{2}{3}%
\end{array}%
\right. ,
\end{equation}
\begin{equation}\label{e7}
j_{n+1}^{(3)}+j_{n}^{(3)}=3J_{n+2}^{(3)},
\end{equation}
\begin{equation}\label{e8}
j_{n}^{(3)}-J_{n+2}^{(3)}=\left\{ 
\begin{array}{ccc}
1 & \textrm{if} & \mymod{n}{0}{3} \\ 
-1 & \textrm{if} & \mymod{n}{1}{3} \\ 
0 & \textrm{if} & \mymod{n}{2}{3}%
\end{array}%
\right. ,
\end{equation}
\begin{equation}\label{e9}
\left( j_{n-3}^{(3)}\right) ^{2}+3J_{n}^{(3)}j_{n}^{(3)}=4^{n},
\end{equation}
\begin{equation}\label{e10}
\sum\limits_{k=0}^{n}J_{k}^{(3)}=\left\{ 
\begin{array}{ccc}
J_{n+1}^{(3)} & \textrm{if} & \mynotmod{n}{0}{3} \\ 
J_{n+1}^{(3)}-1 & \textrm{if} & \mymod{n}{0}{3}%
\end{array}%
\right. 
\end{equation}
and
\begin{equation}\label{e12}
\left( j_{n}^{(3)}\right) ^{2}-9\left( J_{n}^{(3)}\right)^{2}=2^{n+2}j_{n-3}^{(3)},\ n\geq3.
\end{equation}

Using standard techniques for solving recurrence relations, the auxiliary equation, and its roots are given by 
$$x^{3}-x^{2}-x-2=0;\ x = 2,\ \textrm{and}\ x=\frac{-1\pm i\sqrt{3}}{2}.$$ 

Note that the latter two are the complex conjugate cube roots of unity. Call them $\omega_{1}$ and $\omega_{2}$, respectively. Thus the Binet formulas can be written as
\begin{equation}\label{b1}
J_{n}^{(3)}=\frac{2}{7}2^{n}-\frac{3+2i\sqrt{3}}{21}\omega_{1}^{n}-\frac{3-2i\sqrt{3}}{21}\omega_{2}^{n}
\end{equation}
and
\begin{equation}\label{b2}
j_{n}^{(3)}=\frac{8}{7}2^{n}+\frac{3+2i\sqrt{3}}{7}\omega_{1}^{n}+\frac{3-2i\sqrt{3}}{7}\omega_{2}^{n},
\end{equation}
respectively. Now, we use the notation
\begin{equation}\label{h1}
V_{n}^{(2)}=\frac{A\omega_{1}^{n}-B\omega_{2}^{n}}{\omega_{1}-\omega_{2}}=\left\{ 
\begin{array}{ccc}
2 & \textrm{if} & \mymod{n}{0}{3} \\ 
-3 & \textrm{if} & \mymod{n}{1}{3} \\ 
1& \textrm{if} & \mymod{n}{2}{3}
\end{array}%
\right. ,
\end{equation}
where $A=-3-2\omega_{2}$ and $B=-3-2\omega_{1}$. Furthermore, note that for all $n\geq0$ we have 
\begin{equation}\label{im}
V_{n+2}^{(2)}=-V_{n+1}^{(2)}-V_{n}^{(2)},\ V_{0}^{(2)}=2\ \textrm{and}\ V_{1}^{(2)}=-3.
\end{equation}

From the Binet formulas (\ref{b1}), (\ref{b2}) and Eq. (\ref{h1}), we have
\begin{equation}\label{h2}
J_{n}^{(3)}=\frac{1}{7}\left(2^{n+1}-V_{n}^{(2)}\right)\ \textrm{and}\ j_{n}^{(3)}=\frac{1}{7}\left(2^{n+3}+3V_{n}^{(2)}\right).
\end{equation}

On the other hand, the Gelin-Ces\`aro identity \cite[p. 401]{Di} states that
\begin{equation}\label{ge-ce}
F_{n}^{4}-F_{n-2}F_{n-1}F_{n+1}F_{n+2}=1,
\end{equation}
where $F_{n}$ is the classic $n$-th Fibonacci number. Furthermore, Melham and Shannon \cite{Me-Sha} obtained generalizations of the Gelin-Ces\`aro identity. Recently, Sahin \cite{Sa} showed that the Gelin-Ces\`aro identity is satisfied for two families of conditional sequences. 

Motivated by \cite{Me-Sha,Sa}, in this paper, we deal with two families of third-order Jacobsthal sequences. The first family consists of the sequences denoted by $\{J_{n}^{(3)}\}$ and studied in \cite{Cer,Cook-Bac}. We show that the Gelin-Ces\`aro identity is satisfied by the sequence $\{J_{n}^{(3)}\}$. Also, we define a family of generalized third-order Jacobsthal sequences $\{\mathbb{J}_{n}^{(3)}\}$ by the recurrence relation $\mathbb{J}_{n+3}^{(3)}=\mathbb{J}_{n+2}^{(3)}+\mathbb{J}_{n+1}^{(3)}+2\mathbb{J}_{n}^{(3)}$ ($n\geq 0$) with initials conditions $\mathbb{J}_{0}^{(3)}=a$, $\mathbb{J}_{1}^{(3)}=b$ and $\mathbb{J}_{2}^{(3)}=c$, where $a$, $b$ and $c$ are non-zero real numbers. Many sequences in the literature are special cases of this generalized sequence. We find the generating function and Binet formula for the sequence $\{\mathbb{J}_{n}^{(3)}\}_{n\geq0}$. Then, we show that Catalan and Gelin-Ces\`aro identities are satisfied by this generalized sequence.

\section{The first family of third-order Jacobsthal sequences}
\setcounter{equation}{0}

Recently, the authors introduced in \cite{Cook-Bac} a further generalization of the Jacobsthal sequence, namely the third-order Jacobsthal  sequence defined by Eq. (\ref{e2}). Then,
\begin{lemma}[Catalan Identity for $J_{n}^{(3)}$]\label{lem1}
For any nonnegative integers $n$ and $r$, we have
\begin{equation}
\left(J_{n}^{(3)}\right)^{2}-J_{n-r}^{(3)}J_{n+r}^{(3)}=\frac{1}{49}\left\lbrace
\begin{array}{c}
2^{n+1}\left(2^{r}V_{n-r}^{(2)}-2V_{n}^{(2)}+2^{-r}V_{n+r}^{(2)}\right)\\
+7\left(U_{r}^{(2)}\right)^{2}
\end{array}
\right\rbrace,
\end{equation}
where $U_{r}^{(2)}=j_{r-1}^{(3)}-J_{r+1}^{(3)}$.
\end{lemma}
\begin{proof}
From Eqs. (\ref{h1}) and (\ref{h2}), we obtain
\begin{align*}
\left(J_{n}^{(3)}\right)^{2}&-J_{n-r}^{(3)}J_{n+r}^{(3)}\\
&=\frac{1}{49}\left\lbrace
\begin{array}{c}
\left(2^{n+1}-V_{n}^{(2)}\right)^{2}-\left(2^{n-r+1}-V_{n-r}^{(2)}\right)\left(2^{n+r+1}-V_{n+r}^{(2)}\right)
\end{array}
\right\rbrace\\
&=\frac{1}{49}\left\lbrace
\begin{array}{c}
2^{2(n+1)}-2^{n+2}V_{n}^{(2)}+\left(V_{n}^{(2)}\right)^{2}\\
-2^{2(n+1)}+2^{n-r+1}V_{n+r}^{(2)}+2^{n+r+1}V_{n-r}^{(2)}-V_{n-r}^{(2)}V_{n+r}^{(2)}
\end{array}
\right\rbrace\\
&=\frac{1}{49}\left\lbrace
\begin{array}{c}
2^{n+1}\left(2^{r}V_{n-r}^{(2)}-2V_{n}^{(2)}+2^{-r}V_{n+r}^{(2)}\right)\\
+\left(V_{n}^{(2)}\right)^{2}-V_{n-r}^{(2)}V_{n+r}^{(2)}
\end{array}
\right\rbrace\\
&=\frac{1}{49}\left\lbrace
\begin{array}{c}
2^{n+1}\left(2^{r}V_{n-r}^{(2)}-2V_{n}^{(2)}+2^{-r}V_{n+r}^{(2)}\right)+7\left(U_{r}^{(2)}\right)^{2}
\end{array}
\right\rbrace,
\end{align*}
where $U_{r}^{(2)}=j_{r-1}^{(3)}-J_{r+1}^{(3)}$ using Eq. (\ref{e8}). The proof is completed.
\end{proof}

\begin{theorem}[Gelin-Ces\`aro Identity]\label{thm:1}
For any non-negative integers $n\geq 2$, we have
\begin{equation}
\begin{aligned}
\left(J_{n}^{(3)}\right)^{4}&-J_{n-2}^{(3)}J_{n-1}^{(3)}J_{n+1}^{(3)}J_{n+2}^{(3)}\\
&=\frac{1}{7}\left\lbrace
\begin{array}{c}
\left(J_{n}^{(3)}\right)^{2}\left(2+2^{n-1}\left(3R_{n+2}^{(2)}-2R_{n+1}^{(2)}\right)\right)\\
-\frac{1}{7}\left(1+2^{n-1}\left(3R_{n+2}^{(2)}-2R_{n+1}^{(2)}\right)-3\cdot 2^{2n-1}R_{n+1}^{(2)}R_{n+2}^{(2)}\right)
\end{array}
\right\rbrace,
\end{aligned}
\end{equation}
where $W_{n+2}^{(2)}=\frac{1}{7}\left(5V_{n+1}-3V_{n}^{(2)}\right)$ and $V_{n}^{(2)}$ as in Eq. (\ref{h1}).
\end{theorem}
\begin{proof}
For $r=1$ and $r=2$ and Eq. (\ref{im}), we get respectively
\begin{align*}
\left(J_{n}^{(3)}\right)^{2}-J_{n-1}^{(3)}J_{n+1}^{(3)}&=\frac{1}{49}\left(
\begin{array}{c}
2^{n+1}\left(2V_{n-1}^{(2)}-2V_{n}^{(2)}+2^{-1}V_{n+1}^{(2)}\right)+7\left(U_{1}^{(2)}\right)^{2}
\end{array}
\right)\\
&=\frac{1}{49}\left(
\begin{array}{c}
2^{n}\left(4V_{n-1}^{(2)}-4V_{n}^{(2)}+V_{n+1}^{(2)}\right)+7
\end{array}
\right)\\
&=\frac{1}{49}\left(
\begin{array}{c}
7-2^{n}\left(5V_{n}^{(2)}-3V_{n-1}^{(2)}\right)
\end{array}
\right)\\
&=\frac{1}{7}\left(
\begin{array}{c}
1-2^{n}W_{n+1}^{(2)}
\end{array}
\right)
\end{align*}
and
\begin{align*}
\left(J_{n}^{(3)}\right)^{2}-J_{n-2}^{(3)}J_{n+2}^{(3)}&=\frac{1}{49}\left(
\begin{array}{c}
2^{n+1}\left(4V_{n-2}^{(2)}-2V_{n}^{(2)}+2^{-2}V_{n+2}^{(2)}\right)+7\left(U_{2}^{(2)}\right)^{2}
\end{array}
\right)\\
&=\frac{1}{49}\left(
\begin{array}{c}
2^{n-1}\left(16V_{n-2}^{(2)}-8V_{n}^{(2)}+V_{n+2}^{(2)}\right)+7
\end{array}
\right)\\
&=\frac{1}{49}\left(
\begin{array}{c}
7+3\cdot 2^{n-1}\left(5V_{n+1}^{(2)}-3V_{n}^{(2)}\right)
\end{array}
\right)\\
&=\frac{1}{7}\left(
\begin{array}{c}
1+3\cdot 2^{n-1}W_{n+2}^{(2)}
\end{array}
\right)
\end{align*}
 by using Lemma \ref{lem1} and the property $V_{n}^{(2)}=V_{n+3}^{(2)}$ for all $n\geq 0$. Note that the sequence $W_{n}^{(2)}$ satisfies relation $5V_{n+1}^{(2)}-3V_{n}^{(2)}=7W_{n+2}^{(2)}$. So, we obtain
 \begin{align*}
 &J_{n-2}^{(3)}J_{n-1}^{(3)}J_{n+1}^{(3)}J_{n+2}^{(3)}\\
 &=\left( \left(J_{n}^{(3)}\right)^{2}-\frac{1}{7}\left(
\begin{array}{c}
1-2^{n}W_{n+1}^{(2)}
\end{array}
\right)\right)\left( \left(J_{n}^{(3)}\right)^{2}-\frac{1}{7}\left(
\begin{array}{c}
1+3\cdot 2^{n-1}W_{n+2}^{(2)}
\end{array}
\right)\right)\\
&=\left(J_{n}^{(3)}\right)^{4}-\frac{1}{7}\left(J_{n}^{(3)}\right)^{2}\cdot \left(1+3\cdot 2^{n-1}W_{n+2}^{(2)}+1-2^{n}W_{n+1}^{(2)}\right)\\
&\ \ + \frac{1}{49}\left(1-2^{n}W_{n+1}^{(2)}\right)\left(1+3\cdot 2^{n-1}W_{n+2}^{(2)}\right).
\end{align*}
Then, we have
\begin{align*}
 &\left(J_{n}^{(3)}\right)^{4}-J_{n-2}^{(3)}J_{n-1}^{(3)}J_{n+1}^{(3)}J_{n+2}^{(3)}\\
 &=\frac{1}{7}\left\lbrace
\begin{array}{c}
\left(J_{n}^{(3)}\right)^{2}\left(2+2^{n-1}\left(3W_{n+2}^{(2)}-2W_{n+1}^{(2)}\right)\right)\\
-\frac{1}{7}\left(1+2^{n-1}\left(3W_{n+2}^{(2)}-2W_{n+1}^{(2)}\right)-3\cdot 2^{2n-1}W_{n+1}^{(2)}W_{n+2}^{(2)}\right)
\end{array}
\right\rbrace\\
&=\frac{1}{7}\left\{ 
\begin{array}{ccc}
\left\lbrace
\begin{array}{c}
\left(J_{n}^{(3)}\right)^{2}\left(2-11\cdot 2^{n-1}\right)\\
-\frac{1}{7}\left(1-11\cdot 2^{n-1}+9\cdot 2^{2n-1}\right)
\end{array}
\right\rbrace, & \textrm{if} & \mymod{n}{0}{3} \\ 
\left\lbrace
\begin{array}{c}
\left(J_{n}^{(3)}\right)^{2}\left(2+12\cdot 2^{n-1}\right)\\
-\frac{1}{7}\left(1+12\cdot 2^{n-1}+18\cdot 2^{2n-1}\right)
\end{array}
\right\rbrace,  & \textrm{if} & \mymod{n}{1}{3} \\ 
\left\lbrace
\begin{array}{c}
\left(J_{n}^{(3)}\right)^{2}\left(2-2^{n-1}\right)\\
-\frac{1}{7}\left(1-2^{n-1}-6\cdot 2^{2n-1}\right)
\end{array}
\right\rbrace, & \textrm{if} & \mymod{n}{2}{3}
\end{array}%
\right. .
\end{align*}
The proof is completed.
\end{proof}

\section{The second family of generalized third-order Jacobsthal sequences}
\setcounter{equation}{0}

Here, we define a new generalization of the third-order Jacobsthal sequence $\{J_{n}^{(3)}\}_{n\geq0}$. Let us denote this sequence by $\{\mathbb{J}_{n}^{(3)}\}_{n\geq0}$ which is defined recursively by
\begin{equation}\label{n1}
\left\lbrace 
\begin{array}{c}
\mathbb{J}_{n+3}^{(3)}=\mathbb{J}_{n+2}^{(3)}+\mathbb{J}_{n+1}^{(3)}+2\mathbb{J}_{n}^{(3)},\ n\geq 0,\\
\mathbb{J}_{0}^{(3)}=a,\ \mathbb{J}_{1}^{(3)}=b,\ \mathbb{J}_{2}^{(3)}=c,
\end{array}
\right .
\end{equation}
where $a$, $b$ and $c$ are real numbers. For example, the first seven terms of the sequence are $$\{a,\ b,\ c,\ 2a+b+c,\ 2a+3b+2c,\ 4a+4b+5c,\ 10a+9b+9c\}.$$
For $a=0$ and $b=c=1$, we get the ordinary third-order Jacobsthal sequence. Also, when $a=2$, $b=1$ and $c=5$, we get the third-order Jacobsthal-Lucas sequence which is defined in \cite{Cook-Bac}.

In this study, first we obtain the generating function and then Binet's formula for the sequence $\{\mathbb{J}_{n}^{(3)}\}_{n\geq0}$. Finally, we show some properties, for example, the Catalan and Gelin-Ces\`aro identity are satisfied by this sequence.

Now, we can give the generating function of the sequence.
\begin{theorem}[Generating function]\label{thm:2}
The generating function for the sequence $\{\mathbb{J}_{n}^{(3)}\}_{n\geq0}$ is
\begin{equation}
F_{\mathbb{J}}(t)=\frac{a+(b-a)t+(c-b-a)t^{2}}{1-t-t^{2}-2t^{3}}.
\end{equation}
\end{theorem}
\begin{proof}
Let $F_{\mathbb{J}}(t)=\mathbb{J}_{0}^{(3)}+\mathbb{J}_{1}^{(3)}t+\mathbb{J}_{2}^{(3)}t^{2}+\cdots =\sum_{n=0}^{\infty}\mathbb{J}_{n}^{(3)}t^{n}$, which is the formal power series of the generating function for $\{\mathbb{J}_{n}^{(3)}\}$. We obtain that
\begin{align*}
(1-t-t^{2}-2t^{3})F_{\mathbb{J}}(t)&=\mathbb{J}_{0}^{(3)}+\mathbb{J}_{1}^{(3)}t+\mathbb{J}_{2}^{(3)}t^{2}+\cdots \\
&\ \ -\mathbb{J}_{0}^{(3)}t-\mathbb{J}_{1}^{(3)}t^{2}-\mathbb{J}_{2}^{(3)}t^{3}-\cdots \\
&\ \ -\mathbb{J}_{0}^{(3)}t^{2}-\mathbb{J}_{1}^{(3)}t^{3}-\mathbb{J}_{2}^{(3)}t^{4}-\cdots \\
&\ \ -2\mathbb{J}_{0}^{(3)}t^{3}-2\mathbb{J}_{1}^{(3)}t^{4}-2\mathbb{J}_{2}^{(3)}t^{5}-\cdots \\
&=\mathbb{J}_{0}^{(3)}+(\mathbb{J}_{1}^{(3)}-\mathbb{J}_{0}^{(3)})t+(\mathbb{J}_{2}^{(3)}-\mathbb{J}_{1}^{(3)}-\mathbb{J}_{0}^{(3)})t^{2},
\end{align*}
since $\mathbb{J}_{n+3}^{(3)}=\mathbb{J}_{n+2}^{(3)}+\mathbb{J}_{n+1}^{(3)}+2\mathbb{J}_{n}^{(3)}$, $n\geq 0$ and the coefficients of $t^{n}$ for $n\geq 3$ are equal with zero. Then, the theorem is proved.
\end{proof}

In fact, we can give Binet's formula for the sequence as follows.
\begin{theorem}\label{thm:3}
For $n\geq 0$, we have
\begin{equation}\label{p1}
\begin{aligned}
\mathbb{J}_{n}^{(3)}&=\left(\frac{c+b+a}{(2-\omega_{1})(2-\omega_{2})}\right)2^{n}-\left(\frac{c-(2+\omega_{2})b+2\omega_{2}a}{(2-\omega_{1})(\omega_{1}-\omega_{2})}\right)\omega_{1}^{n}\\
&\ \ + \left(\frac{c-(2+\omega_{1})b+2\omega_{1}a}{(2-\omega_{2})(\omega_{1}-\omega_{2})}\right)\omega_{2}^{n}.
\end{aligned}
\end{equation}
\end{theorem}
\begin{proof}
The solution of Eq. (\ref{n1}) is
\begin{equation}\label{e13}
\mathbb{J}_{n}^{(3)}=A_{\mathbb{J}}2^{n}+B_{\mathbb{J}}\omega_{1}^{n}+C_{\mathbb{J}}\omega_{2}^{n}.
\end{equation}
Then, let $\mathbb{J}_{0}^{(3)}=A_{\mathbb{J}}+B_{\mathbb{J}}+C_{\mathbb{J}}$, $\mathbb{J}_{1}^{(3)}=2A_{\mathbb{J}}+B_{\mathbb{J}}\omega_{1}+C_{\mathbb{J}}\omega_{2}$ and $\mathbb{J}_{2}^{(3)}=4A_{\mathbb{J}}+B_{\mathbb{J}}\omega_{1}^{2}+C_{\mathbb{J}}\omega_{2}^{2}$. Therefore, we have $(2-\omega_{1})(2-\omega_{2})A_{\mathbb{J}}=c-(\omega_{1}+\omega_{2})b+\omega_{1}\omega_{2}a$, $(2-\omega_{1})(\omega_{1}-\omega_{2})B_{\mathbb{J}}=c-(2+\omega_{2})b+2\omega_{2}a$, $(2-\omega_{2})(\omega_{1}-\omega_{2})C_{\mathbb{J}}=c-(2+\omega_{1})b+2\omega_{1}a$. Using $A_{\mathbb{J}}$, $B_{\mathbb{J}}$ and $C_{\mathbb{J}}$ in Eq. (\ref{e13}), we obtain
\begin{align*}
\mathbb{J}_{n}^{(3)}&=\left(\frac{c+b+a}{(2-\omega_{1})(2-\omega_{2})}\right)2^{n}-\left(\frac{c-(2+\omega_{2})b+2\omega_{2}a}{(2-\omega_{1})(\omega_{1}-\omega_{2})}\right)\omega_{1}^{n}\\
&\ \ + \left(\frac{c-(2+\omega_{1})b+2\omega_{1}a}{(2-\omega_{2})(\omega_{1}-\omega_{2})}\right)\omega_{2}^{n}.
\end{align*}
The proof is completed.
\end{proof}

\begin{theorem}\label{thm:4}
Assume that $x\neq 0$. We obtain,
\begin{equation}\label{p5}
\sum_{k=0}^{n}\frac{\mathbb{J}_{k}^{(3)}}{x^{k}}=\frac{1}{x^{n}\nu(x)}\left\lbrace
\begin{array}{c}
2\mathbb{J}_{n}^{(3)}+\left(\mathbb{J}_{n+2}^{(3)}-\mathbb{J}_{n+1}^{(3)}\right)x+\mathbb{J}_{n+1}^{(3)}x^{2}\\
-x^{n+1}\left\lbrace \begin{array}{c} c-b-a -\left(a-b\right)x+ax^{2}\end{array}\right\rbrace
\end{array}
\right\rbrace,
\end{equation}
where $\nu(x)=x^{3}-x^{2}-x-2$.
\end{theorem}
\begin{proof}
From Theorem \ref{thm:3}, we have
\begin{align*}
\sum_{k=0}^{n}\frac{\mathbb{J}_{k}^{(3)}}{x^{k}}&=\left(\frac{c+b+a}{(2-\omega_{1})(2-\omega_{2})}\right)\sum_{k=0}^{n}\left(\frac{2}{x}\right)^{k}\\
&\ \ - \left(\frac{c-(2+\omega_{2})b+2\omega_{2}a}{(2-\omega_{1})(\omega_{1}-\omega_{2})}\right)\sum_{k=0}^{n}\left(\frac{\omega_{1}}{x}\right)^{k}\\
&\  \ + \left(\frac{c-(2+\omega_{1})b+2\omega_{1}a}{(2-\omega_{2})(\omega_{1}-\omega_{2})}\right)\sum_{k=0}^{n}\left(\frac{\omega_{2}}{x}\right)^{k}.
\end{align*}
By considering the definition of a geometric sequence, we get
\begin{align*}
\sum_{k=0}^{n}\frac{\mathbb{J}_{k}^{(3)}}{x^{k}}&=\left(\frac{c+b+a}{(2-\omega_{1})(2-\omega_{2})}\right)\frac{2^{n+1}-x^{n+1}}{x^{n}(2-x)}\\
&\ \ - \left(\frac{c-(2+\omega_{2})b+2\omega_{2}a}{(2-\omega_{1})(\omega_{1}-\omega_{2})}\right)\frac{\omega_{1}^{n+1}-x^{n+1}}{x^{n}(\omega_{1}-x)}\\
&\  \ + \left(\frac{c-(2+\omega_{1})b+2\omega_{1}a}{(2-\omega_{2})(\omega_{1}-\omega_{2})}\right)\frac{\omega_{2}^{n+1}-x^{n+1}}{x^{n}(\omega_{2}-x)}\\
&=\frac{1}{x^{n}\nu(x)}\left\lbrace
\begin{array}{c}
A_{\mathbb{J}}(2^{n+1}-x^{n+1})(\omega_{1}-x)(\omega_{2}-x)\\
- B_{\mathbb{J}}(\omega_{1}^{n+1}-x^{n+1})(2-x)(\omega_{2}-x)\\
+ C_{\mathbb{J}}(\omega_{2}^{n+1}-x^{n+1})(2-x)(\omega_{1}-x)
\end{array}
\right\rbrace,
\end{align*}
where 
\begin{equation}\label{abc}
\left\lbrace
\begin{array}{c}
A_{\mathbb{J}}=\frac{c+b+a}{(2-\omega_{1})(2-\omega_{2})},\ B_{\mathbb{J}}=\frac{c-(2+\omega_{2})b+2\omega_{2}a}{(2-\omega_{1})(\omega_{1}-\omega_{2})},\ 
C_{\mathbb{J}}=\frac{c-(2+\omega_{1})b+2\omega_{1}a}{(2-\omega_{2})(\omega_{1}-\omega_{2})}
\end{array}
\right.
\end{equation}
and $\nu(x)=x^{3}-x^{2}-x-2$. Using $\omega_{1}+\omega_{2}=-1$ and $\omega_{1}\omega_{2}=1$, if we rearrange the last equality, then we obtain
\begin{align*}
\sum_{k=0}^{n}\frac{\mathbb{J}_{k}^{(3)}}{x^{k}}&=\frac{1}{x^{n}\nu(x)}\left\lbrace
\begin{array}{c}
A_{\mathbb{J}}(2^{n+1}-x^{n+1})(1+x+x^{2})\\
- B_{\mathbb{J}}(\omega_{1}^{n+1}-x^{n+1})(2\omega_{2}-(2+\omega_{2})x+x^{2})\\
+ C_{\mathbb{J}}(\omega_{2}^{n+1}-x^{n+1})(2\omega_{1}-(2+\omega_{1})x+x^{2})
\end{array}
\right\rbrace\\
&=\frac{1}{x^{n}\nu(x)}\left\lbrace
\begin{array}{c}
A_{\mathbb{J}}2^{n+1}(1+x+x^{2})\\
- B_{\mathbb{J}}\omega_{1}^{n+1}(2\omega_{2}-(2+\omega_{2})x+x^{2})\\
+ C_{\mathbb{J}}\omega_{2}^{n+1}(2\omega_{1}-(2+\omega_{1})x+x^{2})\\
-x^{n+1}\left\lbrace \begin{array}{c} A_{\mathbb{J}}(1+x+x^{2})\\
- B_{\mathbb{J}}(2\omega_{2}-(2+\omega_{2})x+x^{2})\\
+ C_{\mathbb{J}}(2\omega_{1}-(2+\omega_{1})x+x^{2})
\end{array}\right\rbrace
\end{array}
\right\rbrace.
\end{align*}
So, the proof is completed.
\end{proof}

In the following theorem, we give the sum of generalized third-order Jacobsthal sequence corresponding to different indices.
\begin{theorem}\label{thm:5}
For $r\geq m$, we have
\begin{equation}\label{p7}
\sum_{k=0}^{n}\mathbb{J}_{mk+r}^{(3)}=\frac{1}{\sigma_{n}}\left\lbrace
\begin{array}{c}
\mathbb{J}_{m(n+1)+r}^{(3)}-\mathbb{J}_{r}^{(3)}+2^{m}\mathbb{J}_{mn+r}^{(3)}-2^{m}\mathbb{J}_{r-m}^{(3)}\\
-\mathbb{J}_{m(n+1)+r}^{(3)}\mu(m)+\mathbb{J}_{r}^{(3)}\mu(m)\\
+\mathbb{J}_{m(n+2)+r}^{(3)}-\mathbb{J}_{r+m}^{(3)}
\end{array}
\right\rbrace,
\end{equation}
where $\sigma_{n}=2^{m+1}+(1-2^{m})(\omega_{1}^{m}+\omega_{2}^{m})-2$ and $\mu(m)=2^{m}+\omega_{1}^{m}+\omega_{2}^{m}$.
\end{theorem}
\begin{proof}
 Let us take $A_{\mathbb{J}}$, $B_{\mathbb{J}}$ and $C_{\mathbb{J}}$ in Eq. (\ref{abc}). Then, we write
\begin{align*}
\sum_{k=0}^{n}\mathbb{J}_{mk+r}^{(3)}&=A_{\mathbb{J}}2^{r}\sum_{k=0}^{n}2^{mk}-B_{\mathbb{J}}\omega_{1}^{r}\sum_{k=0}^{n}\omega_{1}^{mk}+C_{\mathbb{J}}\omega_{2}^{r}\sum_{k=0}^{n}\omega_{2}^{mk}\\
&=A_{\mathbb{J}}2^{r}\left(\frac{2^{m(n+1)}-1}{2^{m}-1}\right)-B_{\mathbb{J}}\omega_{1}^{r}\left(\frac{\omega_{1}^{m(n+1)}-1}{\omega_{1}^{m}-1}\right)\\
&\ \ +C_{\mathbb{J}}\omega_{2}^{r}\left(\frac{\omega_{2}^{m(n+1)}-1}{\omega_{2}^{m}-1}\right)\\
&=\frac{1}{\sigma_{n}}\left\lbrace
\begin{array}{c}
A_{\mathbb{J}}\left(2^{m(n+1)+r}-2^{r}\right)\left(\omega_{1}^{m}\omega_{2}^{m}-(\omega_{1}^{m}+\omega_{2}^{m})+1\right)\\
- B_{\mathbb{J}}\left(\omega_{1}^{m(n+1)+r}-\omega_{1}^{r}\right)\left(2^{m}\omega_{2}^{m}-(2^{m}+\omega_{2}^{m})+1\right)\\
+ C_{\mathbb{J}}\left(\omega_{2}^{m(n+1)+r}-\omega_{2}^{r}\right)\left(2^{m}\omega_{1}^{m}-(2^{m}+\omega_{1}^{m})+1\right)
\end{array}
\right\rbrace,
\end{align*}
where $\sigma_{n}=2^{m+1}+(1-2^{m})(\omega_{1}^{m}+\omega_{2}^{m})-2$. After some algebra, we obtain
$$\sum_{k=0}^{n}\mathbb{J}_{mk+r}^{(3)}=\frac{1}{\sigma_{n}}\left\lbrace
\begin{array}{c}
\mathbb{J}_{m(n+1)+r}^{(3)}-\mathbb{J}_{r}^{(3)}+2^{m}\mathbb{J}_{mn+r}^{(3)}-2^{m}\mathbb{J}_{r-m}^{(3)}\\
-\mathbb{J}_{m(n+1)+r}^{(3)}\mu(m)+\mathbb{J}_{r}^{(3)}\mu(m)\\
+\mathbb{J}_{m(n+2)+r}^{(3)}-\mathbb{J}_{r+m}^{(3)}
\end{array}
\right\rbrace,$$
where $\mu(m)=2^{m}+\omega_{1}^{m}+\omega_{2}^{m}$. The proof is completed.
\end{proof}

\section{Main results}
\setcounter{equation}{0}

We use the next notation for the Binet formula of generalized third-order Jacobsthal sequence $\mathbb{J}_{n}^{(3)}$. Let 
\begin{equation}\label{h10}
\mathbb{V}_{n}^{(2)}=\frac{A\omega_{1}^{n}-B\omega_{2}^{n}}{\omega_{1}-\omega_{2}}=\left\{ 
\begin{array}{ccc}
c+b-6a & \textrm{if} & \mymod{n}{0}{3} \\ 
2c-5b+2a & \textrm{if} & \mymod{n}{1}{3} \\ 
-3c+4b+4a& \textrm{if} & \mymod{n}{2}{3}
\end{array}
\right. ,
\end{equation}
where $A=(2-\omega_{2})(c-(2+\omega_{2})b+2\omega_{2}a)$ and $B=(2-\omega_{1})(c-(2+\omega_{1})b+2\omega_{1}a)$. Furthermore, note that for all $n\geq0$ we have 
\begin{equation}\label{im2}
\mathbb{V}_{n+2}^{(2)}=-\mathbb{V}_{n+1}^{(2)}-\mathbb{V}_{n}^{(2)},\ \mathbb{V}_{0}^{(2)}=c+b-6a,\ \mathbb{V}_{1}^{(2)}=2c-5b+2a.
\end{equation}
From the Binet formula (\ref{p1}) and Eq. (\ref{h10}), we have
\begin{equation}\label{h20}
\mathbb{J}_{n}^{(3)}=\frac{1}{7}\left(\rho 2^{n}-\mathbb{V}_{n}^{(2)}\right),
\end{equation}
where $\rho=a+b+c$. In particular, if $a=0$ and $b=c=1$, we obtain $\mathbb{J}_{n}^{(3)}=J_{n}^{(3)}$.

\begin{theorem}[Catalan Identity for $\mathbb{J}_{n}^{(3)}$]\label{t1}
For any nonnegative integers $n$ and $r$, we have
\begin{equation}\label{c1}
\left(\mathbb{J}_{n}^{(3)}\right)^{2}-\mathbb{J}_{n-r}^{(3)}\mathbb{J}_{n+r}^{(3)}=\frac{1}{49}\left\lbrace
\begin{array}{c}
2^{n}\rho \left(2^{r}\mathbb{V}_{n-r}^{(2)}-2\mathbb{V}_{n}^{(2)}+2^{-r}\mathbb{V}_{n+r}^{(2)}\right)\\
+7(4a^{2}+3b^{2}+c^{2}-2ac-3bc)\left(U_{r}^{(2)}\right)^{2}
\end{array}
\right\rbrace,
\end{equation}
where $U_{r}^{(2)}=j_{r-1}^{(3)}-J_{r+1}^{(3)}$ and $\rho=a+b+c$.
\end{theorem}
\begin{proof}
From Eq. (\ref{h20}), we obtain
\begin{align*}
\left(\mathbb{J}_{n}^{(3)}\right)^{2}&-\mathbb{J}_{n-r}^{(3)}\mathbb{J}_{n+r}^{(3)}\\
&=\frac{1}{49}\left\lbrace
\begin{array}{c}
\left(\rho 2^{n}-\mathbb{V}_{n}^{(2)}\right)^{2}-\left(\rho 2^{n-r}-\mathbb{V}_{n-r}^{(2)}\right)\left(\rho 2^{n+r}-\mathbb{V}_{n+r}^{(2)}\right)
\end{array}
\right\rbrace\\
&=\frac{1}{49}\left\lbrace
\begin{array}{c}
\rho^{2}2^{2n}-2^{n+1}\rho \mathbb{V}_{n}^{(2)}+\left(\mathbb{V}_{n}^{(2)}\right)^{2}\\
-\rho^{2}2^{2n}+2^{n-r}\rho \mathbb{V}_{n+r}^{(2)}+2^{n+r}\rho \mathbb{V}_{n-r}^{(2)}-\mathbb{V}_{n-r}^{(2)}\mathbb{V}_{n+r}^{(2)}
\end{array}
\right\rbrace\\
&=\frac{1}{49}\left\lbrace
\begin{array}{c}
2^{n}\rho \left(2^{r}\mathbb{V}_{n-r}^{(2)}-2\mathbb{V}_{n}^{(2)}+2^{-r}\mathbb{V}_{n+r}^{(2)}\right)\\
+\left(\mathbb{V}_{n}^{(2)}\right)^{2}-\mathbb{V}_{n-r}^{(2)}\mathbb{V}_{n+r}^{(2)}
\end{array}
\right\rbrace.
\end{align*}
After some algebra, we obtain
\begin{align*}
\left(\mathbb{J}_{n}^{(3)}\right)^{2}-\mathbb{J}_{n-r}^{(3)}\mathbb{J}_{n+r}^{(3)}&=\frac{1}{49}\left\lbrace
\begin{array}{c}
2^{n}\rho \left(2^{r}\mathbb{V}_{n-r}^{(2)}-2\mathbb{V}_{n}^{(2)}+2^{-r}\mathbb{V}_{n+r}^{(2)}\right)\\
+7(4a^{2}+3b^{2}+c^{2}-2ac-3bc)\left(U_{r}^{(2)}\right)^{2}
\end{array}
\right\rbrace,
\end{align*}
where $U_{r}^{(2)}=j_{r-1}^{(3)}-J_{r+1}^{(3)}$ using Eq. (\ref{e8}). The proof is completed.
\end{proof}

\begin{theorem}[Gelin-Ces\`aro Identity]\label{t2}
For any non-negative integers $n\geq 2$, we have
\begin{equation}
\begin{aligned}
&\left(\mathbb{J}_{n}^{(3)}\right)^{4}-\mathbb{J}_{n-2}^{(3)}\mathbb{J}_{n-1}^{(3)}\mathbb{J}_{n+1}^{(3)}\mathbb{J}_{n+2}^{(3)}\\
&=\frac{1}{7}\left\lbrace
\begin{array}{c}
\left(\mathbb{J}_{n}^{(3)}\right)^{2}\left(2\omega+2^{n-2}\rho \left(3\mathbb{W}_{n+2}^{(2)}-2\mathbb{W}_{n+1}^{(2)}\right)\right)\\
-\frac{1}{7}\left(w^{2}+2^{n-2}\rho \omega \left(3\mathbb{W}_{n+2}^{(2)}-2\mathbb{W}_{n+1}^{(2)}\right)-3\cdot 2^{2n-3}\rho^{2} \mathbb{W}_{n+1}^{(2)}\mathbb{W}_{n+2}^{(2)}\right)
\end{array}
\right\rbrace,
\end{aligned}
\end{equation}
where $\mathbb{W}_{n+2}^{(2)}=\frac{1}{7}\left(5\mathbb{V}_{n+1}-3\mathbb{V}_{n}^{(2)}\right)$ and $\mathbb{V}_{n}^{(2)}$ as in Eq. (\ref{h10}).
\end{theorem}
\begin{proof}
From Eq. (\ref{c1}) in Theorem \ref{t1} and $r=1$ and $r=2$, we obtain
\begin{align*}
\left(\mathbb{J}_{n}^{(3)}\right)^{2}-\mathbb{J}_{n-1}^{(3)}\mathbb{J}_{n+1}^{(3)}&=\frac{1}{49}\left\lbrace
\begin{array}{c}
2^{n}\rho \left(2\mathbb{V}_{n-1}^{(2)}-2\mathbb{V}_{n}^{(2)}+2^{-1}\mathbb{V}_{n+1}^{(2)}\right)\\
+7(4a^{2}+3b^{2}+c^{2}-2ac-3bc)
\end{array}
\right\rbrace\\
&=\frac{1}{49}\left\lbrace
\begin{array}{c}
2^{n-1}\rho \left(4\mathbb{V}_{n-1}^{(2)}-4\mathbb{V}_{n}^{(2)}+\mathbb{V}_{n+1}^{(2)}\right)\\
+7(4a^{2}+3b^{2}+c^{2}-2ac-3bc)
\end{array}
\right\rbrace\\
&=\frac{1}{49}\left\lbrace
\begin{array}{c}
7(4a^{2}+3b^{2}+c^{2}-2ac-3bc)\\
-2^{n-1}\rho \left(5\mathbb{V}_{n}^{(2)}-3\mathbb{V}_{n-1}^{(2)}\right)
\end{array}
\right\rbrace\\
&=\frac{1}{7}\left(
\begin{array}{c}
\omega-2^{n-1}\rho \mathbb{W}_{n+1}^{(2)}
\end{array}
\right)
\end{align*}
and
\begin{align*}
\left(\mathbb{J}_{n}^{(3)}\right)^{2}-\mathbb{J}_{n-2}^{(3)}\mathbb{J}_{n+2}^{(3)}&=\frac{1}{7}\left(
\begin{array}{c}
\omega+3\cdot 2^{n-2}\rho \mathbb{W}_{n+2}^{(2)}
\end{array}
\right),
\end{align*}
where $5\mathbb{V}_{n}^{(2)}-3\mathbb{V}_{n-1}^{(2)}=7\mathbb{W}_{n+1}^{(2)}$ and $\omega=4a^{2}+3b^{2}+c^{2}-2ac-3bc$. So, we can write the next equality:
 \begin{align*}
 &\mathbb{J}_{n-2}^{(3)}\mathbb{J}_{n-1}^{(3)}\mathbb{J}_{n+1}^{(3)}\mathbb{J}_{n+2}^{(3)}\\
 &=\left( \left(\mathbb{J}_{n}^{(3)}\right)^{2}-\frac{1}{7}\left(
\begin{array}{c}
\omega-2^{n-1}\rho \mathbb{W}_{n+1}^{(2)}
\end{array}
\right)\right)\left( \left(\mathbb{J}_{n}^{(3)}\right)^{2}-\frac{1}{7}\left(
\begin{array}{c}
\omega+3\cdot 2^{n-2}\rho \mathbb{W}_{n+2}^{(2)}
\end{array}
\right)\right)\\
&=\left\lbrace \begin{array}{c}
\left(\mathbb{J}_{n}^{(3)}\right)^{4}-\frac{1}{7}\left(\mathbb{J}_{n}^{(3)}\right)^{2}\cdot \left(2\omega +3\cdot 2^{n-2}\rho \mathbb{W}_{n+2}^{(2)}-2^{n-1}\rho \mathbb{W}_{n+1}^{(2)}\right)\\
+\frac{1}{49}\left(\omega-2^{n-1}\rho \mathbb{W}_{n+1}^{(2)}\right)\left(\omega+3\cdot 2^{n-2}\rho \mathbb{W}_{n+2}^{(2)}\right)
\end{array}
\right\rbrace.
\end{align*}
Then, we have
\begin{align*}
 &\left(\mathbb{J}_{n}^{(3)}\right)^{4}-\mathbb{J}_{n-2}^{(3)}\mathbb{J}_{n-1}^{(3)}\mathbb{J}_{n+1}^{(3)}\mathbb{J}_{n+2}^{(3)}\\
 &=\frac{1}{7}\left\lbrace
\begin{array}{c}
\left(\mathbb{J}_{n}^{(3)}\right)^{2}\left(2\omega+2^{n-2}\rho \left(3\mathbb{W}_{n+2}^{(2)}-2\mathbb{W}_{n+1}^{(2)}\right)\right)\\
-\frac{1}{7}\left(w^{2}+2^{n-2}\rho \omega \left(3\mathbb{W}_{n+2}^{(2)}-2\mathbb{W}_{n+1}^{(2)}\right)-3\cdot 2^{2n-3}\rho^{2} \mathbb{W}_{n+1}^{(2)}\mathbb{W}_{n+2}^{(2)}\right)
\end{array}
\right\rbrace.
\end{align*}
The proof is completed.
\end{proof}

By the aid of the last theorem we have the following corollary.
\begin{corollary}
For any non-negative integers $n\geq 2$, we have
\begin{align*}
 &\left(\mathbb{J}_{n}^{(3)}\right)^{4}-\mathbb{J}_{n-2}^{(3)}\mathbb{J}_{n-1}^{(3)}\mathbb{J}_{n+1}^{(3)}\mathbb{J}_{n+2}^{(3)}\\
&=\frac{1}{7}\left\{ 
\begin{array}{ccc}
\left\lbrace
\begin{array}{c}
\left(\mathbb{J}_{n}^{(3)}\right)^{2}\left(2\omega+2^{n-2}\rho A\right)\\
-\frac{1}{7}\left(w^{2}+2^{n-2}\rho \omega A-3\cdot 2^{2n-3}\rho^{2} \mathbb{T}_{n}^{(2)}\right)
\end{array}
\right\rbrace, & \textrm{if} & \mymod{n}{0}{3} \\ 
\left\lbrace
\begin{array}{c}
\left(\mathbb{J}_{n}^{(3)}\right)^{2}\left(2\omega+2^{n-2}\rho B\right)\\
-\frac{1}{7}\left(w^{2}+2^{n-2}\rho \omega B-3\cdot 2^{2n-3}\rho^{2} \mathbb{T}_{n}^{(2)}\right)
\end{array}
\right\rbrace,  & \textrm{if} & \mymod{n}{1}{3} \\ 
\left\lbrace
\begin{array}{c}
\left(\mathbb{J}_{n}^{(3)}\right)^{2}\left(2\omega+2^{n-2}\rho C\right)\\
-\frac{1}{7}\left(w^{2}+2^{n-2}\rho \omega C-3\cdot 2^{2n-3}\rho^{2} \mathbb{T}_{n}^{(2)}\right)
\end{array}
\right\rbrace, & \textrm{if} & \mymod{n}{2}{3}
\end{array}
\right. ,
\end{align*}
where $$\left\{ 
\begin{array}{c}
A=-c-10b+24a,\ B=-11c+23b-2a,\ C=12c-13b-22a,
\end{array}\right. 
$$ and $$\mathbb{T}_{n}^{(2)}=\left\{ 
\begin{array}{ccc}
(2c-b-6a)(c-4b+4a), & \textrm{if} & \mymod{n}{0}{3} \\ 
(c-4b+4a)(-3c+5b+2a),  & \textrm{if} & \mymod{n}{1}{3} \\ 
(-3c+5b+2a)(2c-b-6a), & \textrm{if} & \mymod{n}{2}{3}
\end{array}
\right. .$$
\end{corollary}

\medskip

\begin{thebibliography}{99}


\bibitem[Ba]{Ba}
P. Barry, Triangle geometry and Jacobsthal numbers, \emph{Irish Math. Soc. Bull.}, \textbf{51}, (2003), 45--57.
\bibitem[Ce]{Cer} 
G. Cerda-Morales, Identities for Third Order Jacobsthal Quaternions, \emph{Advances in Applied Clifford Algebras}, \textbf{27(2)}, (2017), 1043--1053.
\bibitem[Ce1]{Cer1} 
G. Cerda-Morales, On a Generalization of Tribonacci Quaternions, \emph{Mediterranean Journal of Mathematics}, \textbf{14:239} (2017), 1--12.
\bibitem[Cook-Bac]{Cook-Bac} 
C. K. Cook and M. R. Bacon, Some identities for Jacobsthal and Jacobsthal-Lucas numbers satisfying higher order recurrence relations, \emph{Annales Mathematicae et Informaticae}, \textbf{41}, (2013), 27--39.
\bibitem[Di]{Di} 
L. E. Dickson, History of the Theory of Numbers, Vol. I, Chelsea Publishing Co., New York, 1966.
\bibitem[Ho2]{Hor2} 
A. F. Horadam, Jacobsthal and Pell Curves, \emph{The Fibonacci Quarterly}, \textbf{26(1)}, (1988), 79--83.
\bibitem[Ho3]{Hor3} 
A. F. Horadam, Jacobsthal representation numbers, \emph{The Fibonacci Quarterly}, \textbf{34(1)}, (1996), 40--54.
\bibitem[Me-Sha]{Me-Sha} 
R. S. Melham and A. G. Shannon, A generalization of the Catalan identity and some consequences, \emph{The Fibonacci Quarterly}, \textbf{33}, (1995), 82--84.
\bibitem[Sa]{Sa} 
M. Sahin, The Gelin-Ces\`aro identity in some conditional sequences, \emph{Hacettepe Journal of Mathematics and Statistics}, \textbf{40(6)}, (2011), 855--861.
\end{thebibliography}
\end{document}